\documentclass{amsart}
\usepackage[ansinew]{inputenc}
\usepackage[english]{babel}
\usepackage{amsmath,latexsym,amssymb,verbatim}
\usepackage{pdfsync}

\usepackage[all]{xy}

\usepackage{hyperref}

\newtheorem{theorem}{Theorem}[section]
\newtheorem{lemma}[theorem]{Lemma}
\newtheorem{proposition}[theorem]{Proposition}
\newtheorem{corollary}[theorem]{Corollary}
\newtheorem{remark}[theorem]{Remark}
\newtheorem{note}[theorem]{Note}

\newtheorem{Formula of adjoint functors}[theorem]{Formula of adjoint functors}

\newtheorem{example}[theorem]{Example}
\newtheorem{examples}[theorem]{Examples}
\newtheorem{definition}[theorem]{Definition}
\newtheorem{notation}[theorem]{Notation}

\newtheorem{Adjunction formula}[theorem]{\indent\sc Adjunction formula}

\DeclareMathOperator{\limi}{{lim}}
\newcommand{\ilim}[1]{\,\underset{#1}{\underset{\to}{\limi}}\,}
\newcommand{\plim}[1]{\,\underset{#1}{\underset{\leftarrow}{\limi}}\,}

\DeclareMathOperator{\Hom}{{Hom}}
\DeclareMathOperator{\Spec}{{Spec}}
\DeclareMathOperator{\Coker}{{Coker}}
\DeclareMathOperator{\Ker}{{Ker}}
\DeclareMathOperator{\Ima}{{Im}}

\NoCompileMatrices

\input arrow.tex

\begin{document}

\title{Functorial characterizations of Mittag-Leffler modules}

\author{Carlos Sancho, Fernando Sancho and Pedro Sancho}
%\address[José Navarro]{Departamento de Matemáticas, Universidad de Extremadura,
%Avenida de Elvas s/n, 06071 Badajoz, Spain}
%\email{navarrogarmendia@unex.es}
%
%
%
%\author{Pedro Sancho}
%\address[Pedro Sancho]{Departamento de Matemáticas, Universidad de Extremadura,
%Avenida de Elvas s/n, 06071 Badajoz, Spain}
%\email{sancho@unex.es}
%\thanks{Corresponding Author: Pedro Sancho.}

\date{Junio, 2017}

\begin{abstract} We give some functorial characterizations of Mittag-Leffler modules and strict Mittag-Leffler modules.

\end{abstract}

\maketitle

\section{Introduction}
Let $R$ be a commutative (associative with unit) ring. 
Let $\mathcal R$ be the covariant functor from the category of commutative $R$-algebras to the category of  rings
defined by  $\mathcal R(S):=S$ for any commutative  $R$-algebra $S$.
Let $M$ be an  $R$-module. Consider the functor of $\mathcal R$-modules, $\mathcal M$, defined by $\mathcal M(S):=M\otimes_R S$, for any commutative $R$-algebra $S$. $\mathcal M$ is said to be the functor of  {\sl quasi-coherent}  $\mathcal R$-modules associated with $M$. It is easy to prove that the category of $R$-modules is equivalent to the category of functors of quasi-coherent $\mathcal R$-modules. Consider the dual functor
$\mathcal M^*:=\mathbb Hom_{\mathcal R}(\mathcal M,\mathcal R)$ defined by
$\mathcal M^*(S):=\Hom_S(M\otimes_RS,S)$.  $\mathcal M^*$ is called an $\mathcal R$-module scheme.  In general, the canonical morphism  $M\to M^{**}$ is not an isomorphism, but, surprisingly,  $\mathcal M=\mathcal M^{**}$ (see \ref{reflex}). 
This result has many applications in Algebraic Geometry
 (see \cite{Pedro}), for example the Cartier duality of commutative affine groups and commutative formal groups. 

 In  \cite{Amel2}, we proved that an
 $R$-module $M$ is a projective module of finite type iff  $\mathcal M$ is a module scheme. In  \cite{Pedro2}, we proved that  $M$ is a flat $R$-module iff $\mathcal M$ is a direct limit of module schemes. It is also proved that $M$ is a flat Mittag-Leffler module iff $\mathcal M$ is the direct limit of its submodule schemes. In \cite{Pedro3}, we proved that $M$ is a flat strict Mittag-Leffler module (see \cite{Garfinkel}, for definition and properties) iff $\mathcal M$ is the direct limit of its submodule schemes, $\mathcal M=\ilim{i}\mathcal N_i^*$,  and the morphisms $\mathcal M^*\to\mathcal N_i$ are epimorphisms.

 The definition of a Mittag-Leffler module  is slightly elaborated (see \cite[Tag 0599]{stacks-project}). Mittag-Leffler conditions were first introduced by Grothendieck in \cite{Grot}, and deeply studied
by some authors, such as  Raynaud and Gruson in \cite{RaynaudG}. Every module is a direct limit of finitely presented modules.  Roughly speaking, we  prove that a module $M$ is a Mittag-Leffler module iff $\mathcal M$ is a direct limit of finitely presented functors of submodules.

Let $P$ be an $R$-module. $P$ is a finitely presented module iff $$\Hom_R(P,\ilim{i} N_i)=\ilim{i}\Hom_R(P, N_i),$$ for any direct system $\{N_i\}$ of $R$-modules (see \ref{4.4}). We will say that a functor of $\mathcal R$-modules, $\mathbb P$,  is an FP-functor if $$\Hom_{\mathcal R}(\mathbb  P,\ilim{i}\mathcal N_i)=\ilim{i}\Hom_{\mathcal R}(\mathbb  P,\mathcal N_i)$$ for every direct system of quasi-coherent modules $\{\mathcal N_i\}$.  Module schemes are FP-functors. We prove the following theorem.

\begin{theorem}   $\mathbb M$ is an FP-functor of $\mathcal R$-modules iff $\mathbb M^*$ is the cokernel of a morphism $F\colon \oplus_{i} \mathcal P_i^*\to \oplus_{j} \mathcal Q_j^*$, where $P_i,Q_j$ are finitely presented $R$-modules, for every $i,j$. 
\end{theorem}

Let $\{N_i\}$ be the set of the finitely generated submodules of $M$. Let ${\tilde N_i}:=\Ima[\mathcal N_i\to\mathcal M]$, for any $N_i$. Then, $\mathcal M=\ilim{i} {\tilde N_i}.$
We prove the following theorems.

\begin{theorem} Let $M$ be an $R$-module. The following statements are equivalent:\begin{enumerate}

\item $M$ is a  Mittag-Leffler module.

\item ${\tilde N_i}$ is an FP-functor, for any $i$.

\item 
$\mathcal M$ is a direct limit of FP-functors of $\mathcal R$-submodules. 

\item  The kernel of every morphism
$\mathcal R^n\to\mathcal M$  is isomorphic to a quotient of a module scheme.

\item The kernel of every morphism
$\mathcal N^*\to\mathcal M$ is isomorphic to a quotient of a module scheme, for any  $R$-module $N$.

\end{enumerate}

\end{theorem}

\begin{theorem}  Let $M$ be an $R$-module. The following statements are equivalent:\begin{enumerate}

\item $M$ is a  strict Mittag-Leffler module.

\item ${\tilde N_i}$  is an FP-functor and the natural morphism $\mathcal M^*\to {{\tilde N_i}}^*$ is an epimorphism, for any $i$.

\item $\mathcal M$ is an $\mathcal R$-submodule 
of some $\mathcal R$-module $\prod_r \mathcal P_r$, where $P_r$ is a finitely presented modules, for every $r$. 

\item  The cokernel of every morphism
$\mathcal M^*\to\mathcal R^n$  is isomorphic  to an $\mathcal R$-submodule of a quasi-coherent module.

\item The cokernel of every morphism
$\mathcal M^*\to\mathcal N$ is isomorphic to an $\mathcal R$-submodule of a quasi-coherent module, for any  $R$-module $N$.

\end{enumerate}

\end{theorem}

\section{Preliminaries}

Let $R$ be a commutative ring (associative with a unit). All the  functors considered in this paper are covariant functors from the category of commutative $R$-algebras (always assumed to be associative with a unit) to the category of sets. A functor $\mathbb X$ is said to be a functor of sets (resp. groups, rings,  etc.) if $\mathbb X$ is a functor from the category of  commutative $R$-algebras to the category of sets (resp. groups, rings, etc.).

\begin{notation} \label{nota2.1}For simplicity,
we shall sometimes use $x \in \mathbb X$  to denote $x \in \mathbb X(S)$. Given $x \in \mathbb X(S)$ and a morphism of commutative $R$-algebras $S \to S'$, we shall still denote by $x$ its image by the morphism $\mathbb X(S) \to \mathbb X(S')$.\end{notation}

 An $\mathcal R$-module $\mathbb M$ is a functor of abelian groups endowed with a morphism of functors
$$\mathcal R\times \mathbb M\to\mathbb M$$ satisfying the module axioms (in other words, the morphism $\mathcal R\times \mathbb M\to\mathbb M$ yields an $S$-module structure on $\mathbb M(S)$ for any commutative $R$-algebra $S$).
Let $\mathbb M$ and $\mathbb M'$ be two $\mathcal R$-modules.
 A morphism of $\mathcal R$-modules $f\colon \mathbb M\to \mathbb M'$
 is a morphism of functors  such that the  morphism $f_{S}\colon \mathbb M({S})\to
 \mathbb M'({S})$ defined by $f$ is a morphism of ${S}$-modules, for any commutative $R$-algebra ${S}$.
 We shall denote by $\Hom_{\mathcal R}(\mathbb M,\mathbb M')$ the  family of all the morphisms of $\mathcal R$-modules from $\mathbb M$ to $\mathbb M'$.

\begin{remark} Direct limits, inverse limits of $\mathcal R$-modules and  kernels, cokernels, images, etc.,  of morphisms of $\mathcal R$-modules are regarded in the category of $\mathcal R$-modules.\end{remark}

One has
$$\aligned & 
(\Ker f)({S})=\Ker f_{S},\, (\Coker f)({S})=\Coker f_{S},\, (\Ima f)({S})=\Ima f_{S},\\
&  (\ilim{i\in I} \mathbb M_i)({S})=\ilim{i\in I} (\mathbb M_i({S})),\,
(\plim{j\in J} \mathbb M_j)({S})=\plim{j\in J} (\mathbb M_j({S})),\endaligned $$
(where $I$ is an upward directed set and $J$ a downward directed set).
$\mathbb M\otimes_{\mathcal R}\mathbb M'$ is defined by $(\mathbb M\otimes_{\mathcal R}\mathbb M')({S}):=\mathbb M({S})\otimes_{{S}}\mathbb M'({S})$, for any commutative $R$-algebra ${S}$.

\begin{definition}  Given an $\mathcal R$-module $\mathbb M$ and a commutative $R$-algebra ${S}$,   we shall denote by $\mathbb M_{|{S}}$  the restriction of $\mathbb M$ to the category of commutative ${S}$-algebras, i.e.,  $$\mathbb M_{\mid {S}}(S'):=\mathbb M(S'),$$ for any commutative ${S}$-algebra $S'$.
\end{definition}

We shall denote by ${\mathbb Hom}_{\mathcal R}(\mathbb M,\mathbb M')$\footnote{In this paper, we shall only  consider well-defined functors ${\mathbb Hom}_{\mathcal R}(\mathbb M,\mathbb M')$, that is to say, functors such that $\Hom_{\mathcal {S}}(\mathbb M_{|{S}},\mathbb {M'}_{|{S}})$ is a set, for any ${S}$.} the  $\mathcal R$-module defined by $${\mathbb Hom}_{\mathcal R}(\mathbb M,\mathbb M')({S}):={\rm Hom}_{\mathcal {S}}(\mathbb M_{|{S}}, \mathbb M'_{|{S}}).$$  Obviously,
$$(\mathbb Hom_{\mathcal R}(\mathbb M,\mathbb M'))_{|{S}}=
\mathbb Hom_{\mathcal {S}}(\mathbb M_{|{S}},\mathbb M'_{|{S}}).$$

\begin{notation} \label{nota2.2} Let $\mathbb M$ be an $\mathcal R$-module. We shall denote $\mathbb M^*=\mathbb Hom_{\mathcal R}(\mathbb M,\mathcal R)$.\end{notation}

\begin{proposition} \label{trivial} Let $\mathbb M$ and $\mathbb N$ be two  $\mathcal R$-modules. Then,
$$\Hom_{\mathcal R}(\mathbb M,\mathbb N^*)=\Hom_{\mathcal R}(\mathbb N,\mathbb M^*),\, f\mapsto \tilde f,$$
where $\tilde f$ is defined as follows: $\tilde f(n)(m):=f(m)(n)$, for any $m\in\mathbb M$ and $n\in\mathbb N$.

\end{proposition}

\begin{proof} $\Hom_{\mathcal R}(\mathbb M,\mathbb N^*)=\Hom_{\mathcal R}(\mathbb M\otimes_{\mathcal R}\mathbb N,\mathcal R)=\Hom_{\mathcal R}(\mathbb N,\mathbb M^*)$.

\end{proof}

\begin{proposition} \cite[1.15]{Amel}  \label{adj2} Let $\mathbb M$ be an $\mathcal R$-module, ${S}$ a commutative $R$-algebra and $N$ an ${S}$-module.  Then,
$$\Hom_{\mathcal {S}}(\mathbb M_{|{S}},\mathcal N)=\Hom_{\mathcal R}(\mathbb M,\mathcal N).$$
In particular,
$$\mathbb M^*({S})=\Hom_{\mathcal R}(\mathbb M,\mathcal {S}).$$
\end{proposition}

\subsection{Quasi-coherent modules}

\begin{definition} Let $M$ (resp. $N$,  $V$, etc.) be an $R$-module. We shall denote by  ${\mathcal M}$  (resp. $\mathcal N$, $\mathcal V$, etc.)   the $\mathcal R$-module defined by ${\mathcal M}({S}) := M \otimes_R {S}$ (resp. $\mathcal N({S}):=N\otimes_R {S}$,  $\mathcal V({S}):=V\otimes_R {S}$, etc.). $\mathcal M$  will be called the quasi-coherent $\mathcal R$-module associated with $M$. 
\end{definition}

 ${\mathcal M}_{\mid {S}}$ is the quasi-coherent $\mathcal {S}$-module associated with $M \otimes_R
{S}$.  For any pair of $R$-modules $M$ and $N$, the quasi-coherent module associated with $M\otimes_R N$ is $\mathcal M\otimes_{\mathcal R}\mathcal N$.

\begin{proposition} \cite[1.12]{Amel}  The functors
$$\aligned \text{Category of $R$-modules } & \to \text{ Category of quasi-coherent $\mathcal R$-modules }\\ M & \mapsto \mathcal M\\ \mathcal M(R) & \leftarrow\!\shortmid \mathcal M\endaligned$$
stablish an equivalence  of categories. In particular,
$${\rm Hom}_{\mathcal R} ({\mathcal M},{\mathcal M'}) = {\rm Hom}_R (M,M').$$
\end{proposition}

Let $f\colon M\to N$ be a morphism of $R$-modules and $\tilde f\colon \mathcal M\to \mathcal N$
the associated morphism of $\mathcal R$-modules. Let $C=\Coker f$, then $\Coker\tilde f=\mathcal C$, which is a quasi-coherent module.

\begin{proposition} \cite[1.3]{Amel}\label{tercer}
For every  ${\mathcal R}$-module $\mathbb M$ and every $R$-module $M$, it is satisfied that
$${\rm Hom}_{\mathcal R} ({\mathcal M}, \mathbb M) = {\rm Hom}_R (M, \mathbb M(R)),\, f\mapsto f_R.$$
\end{proposition}

\begin{notation} Let $\mathbb M$ be an $\mathcal R$-module. 
We shall denote by $\mathbb M_{qc}$ the quasi-coherent module associated with
the $R$-module $\mathbb M(R)$, that is, $$\mathbb M_{qc}({S}):=\mathbb M(R)\otimes_R{S}.$$\end{notation}

\begin{proposition} \label{tercerb} For each  $\mathcal R$-module $\mathbb M$ one has the natural morphism $$\mathbb M_{qc}\to \mathbb M, \,m\otimes s\mapsto s\cdot m,$$ for any $m\otimes s\in \mathbb M_{qc}({S})=\mathbb M(R)\otimes_R {S}$, and a functorial equality 
$$\Hom_{\mathcal R}(\mathcal N,\mathbb M_{qc}))=\Hom_{\mathcal R}(\mathcal N,\mathbb M),$$
for any quasi-coherent $\mathcal R$-module $\mathcal N$.
\end{proposition}

\begin{proof} Observe that $\Hom_{\mathcal R}(\mathcal N,\mathbb M)\overset{\text{\ref{tercer}}}=\Hom_R(N,\mathbb M(R))\overset{\text{\ref{tercer}}}=\Hom_{\mathcal R}(\mathcal N,\mathbb M_{qc})$.

\end{proof}

Obviously, an $\mathcal R$-module $\mathbb M$ is a quasi-coherent module iff
the natural morphism $\mathbb M_{qc}\to \mathbb M$ is an isomorphism.

\begin{theorem}  \cite[1.8]{Amel}\label{prop4} 
Let $M$ and $M'$ be $R$-modules. Then, $${\mathcal M} \otimes_{\mathcal R} {\mathcal M'}={\mathbb Hom}_{\mathcal R} ({\mathcal M^*}, {\mathcal M'}),\, m\otimes m'\mapsto \tilde{m\otimes m'},$$
where $ \tilde{m\otimes m'}(w):=w(m)\cdot m'$, for any $w\in \mathcal M^*$.
\end{theorem}

\begin{note} \label{2.12N} In particular, ${\Hom}_{\mathcal R} ({\mathcal M^*}, {\mathcal M'})={M} \otimes_{R} { M'}$, and it is easy to prove that the morphism $f=\sum_i m_i\otimes m'_i\in {\Hom}_{\mathcal R} ({\mathcal M^*}, {\mathcal M'})={M} \otimes_{R} { M'}$ factors through the quasi-coherent module associated with the submodule $\langle m'_i\rangle \subseteq M'$.

\end{note}

%\begin{proof} ${\Hom}_{\mathcal R} ({\mathcal M^*}, {\mathcal M'}) \overset{\text{\ref{L5.111}}}= \overline{\mathcal M'}(M)={M} \otimes_{R} {M'}.$
%
%
%\end{proof}

If we make $\mathcal M'=\mathcal R$ in the previous theorem, we obtain the following theorem.

\begin{theorem} \cite[II,\textsection 1,2.5]{gabriel}  \cite[1.10]{Amel}\label{reflex}
Let $M$ be an $R$-module. Then, $${\mathcal M}={\mathcal M^{**}}.$$
\end{theorem}

\begin{definition} Let $\mathbb M$ be  an $\mathcal R$-module. We shall say that
$\mathbb M^*$ is a dual functor.
We shall say that  an $\mathcal R$-module  ${\mathbb M}$ is reflexive if ${\mathbb M}={\mathbb M}^{**}$.\end{definition}

\begin{example}  Quasi-coherent modules are reflexive.\end{example}

\subsection{$\mathcal R$-module schemes} 

\begin{definition} Let $M$ be an $R$-module. $\mathcal M^*$  will be called the $\mathcal R$-module scheme associated with $M$.
\end{definition}

\begin{definition} Let $\mathbb N$ be an $\mathcal R$-module. 
We shall denote by $\mathbb N_{sch}$ the $\mathcal R$-module scheme defined by $$\mathbb N_{sch}:=((\mathbb N^*)_{qc})^*.$$\end{definition}

\begin{proposition} \label{1211} Let $\mathbb N$ be a functor of $R$-modules. Then,  we have a canonical morphism $\mathbb N\to \mathbb N_{sch}$ and
\begin{enumerate}
\item $\Hom_{\mathcal R}(\mathbb N,\mathcal M^*)=
\Hom_{\mathcal R}(\mathbb N_{sch},\mathcal M^*),$
for any module scheme $\mathcal M^*$.

\item $
\Hom_{\mathcal R}(\mathbb N_{sch},\mathcal M)=\mathbb N^*(R)\otimes_R M$,
for any quasi-coherent module $\mathcal M$.

\end{enumerate}

\end{proposition}

\begin{proof} $\Hom _{\mathcal R}(\mathbb N,\mathcal M^*) \! \overset{\text{\ref{trivial}}} =\!
\Hom_{\mathcal R}(\mathcal M,\mathbb N^*)\! \overset{\text{\ref{tercerb}}}=\!\Hom_{\mathcal R}(\mathcal M, (\mathbb N^*)_{qc}){\overset{\text{\ref{trivial}}}=
\Hom_{\mathcal R}(\mathbb N_{sch},\mathcal M^*)},$
and $
\Hom_{\mathcal R}(\mathbb N_{sch},\mathcal M)\overset{\text{\ref{prop4} }}=(\mathbb N^*)_{qc}(R)\otimes_RM=\mathbb N^*(R)\otimes_R M$.

\end{proof}

 Let $\{U_i\}_{i\in I}$ be an open covering of a scheme $X$. We shall say that the obvious morphism $Y=\coprod_{i\in I} U_i\to X$ is an open covering.

\begin{definition} Let $\mathbb F$ be a functor of sets.  
$\mathbb F$  is said to be a sheaf in the Zariski topos if 
for any commutative $R$-algebra ${S}$ and any  open covering $\Spec {S}_1\to \Spec {S}$, the sequence of morphisms
$$\xymatrix{ \mathbb F({S})\ar[r] &  \mathbb F({S}_1)  \ar@<1ex>[r]  \ar@<-1ex>[r] & 
\mathbb F({S}_1\otimes_{S} {S}_1)}$$
is exact.

\end{definition}

\begin{example} $\mathcal M$ is a sheaf in the Zariski topos. 

\end{example}

Let $\mathbb F_1$ and $\mathbb F_2$ be sheaves in the Zariski topos.  If $f\colon \mathbb F_1\to \mathbb F_2$ is a morphism of $\mathcal R$-modules, it is easy to check that $\Ker f$ is a sheaf in the Zariski topos.

\begin{theorem} \cite[1.28]{Pedro2} \label{L5.11} Let $\{\mathbb F_i\}$ be a direct system of sheaves of $\mathcal R$-modules. Then,
$$\Hom_{\mathcal R}(\mathcal N^*,\ilim{i} \mathbb F_i)=\ilim{i}\Hom_{\mathcal R}(\mathcal N^*, \mathbb F_i).$$

\end{theorem}

\subsection{From the category of $R$-algebras to the category of $R$-modules}

\begin{notation} \label{notation} Let $\mathbb F=\mathbb N^*$ be a dual functor of $\mathcal R$-modules. We can consider the following functor from the category of $R$-modules to the category of $R$-modules 
$$\bar {\mathbb F}(N):=\Hom_{\mathcal R}(\mathbb N,\mathcal N)\overset{\text{\ref{trivial}}}=\Hom_{\mathcal R}(\mathcal N^*,\mathbb F),$$
for any $R$-module $N$. 

\end{notation}

\begin{examples} $\overline{\mathcal M^*}(N)=\Hom_R(M,N)$. $\bar{\mathcal M}(N)\overset{\text{\ref{prop4}}}=M\otimes_R N$.

\end{examples}

Observe that $\bar{\mathbb F}(S)=\mathbb F(S)$, for any commutative $R$-algebra $S$. Given an $R$-module $N$, consider the $R$-algebra $R\oplus N$, where $(r,n)\cdot (r',n'):=(rr',rn'+r'n)$. It is easy to check that

$$\mathbb F(R\oplus N)=\bar{\mathbb F}(R)\oplus \bar{\mathbb F}(N),$$
and $\bar{\mathbb F}(N)=\Ker({\mathbb F}(R\oplus N)\to {\mathbb F}(R))$.

Let ${\mathbb F}$, ${\mathbb F'}$ and  ${\mathbb F''}$ be  dual functors of $\mathcal R$-modules. Then, ${\mathbb F}\to {\mathbb F'}\to {\mathbb F''}$ is an exact sequence of morphisms of $\mathcal R$-modules iff $\bar{\mathbb F}(N)\to \bar{\mathbb F'}(N)\to \bar{\mathbb F''}(N)$ is an exact sequence of morphisms of $R$-modules, for any $R$-module $N$.

\begin{lemma} \label{main} The obvious morphism
$$\Hom_{\mathcal R}(\prod_{i\in I} \mathcal  R,\mathcal N)\to \Hom_{\mathcal R}(\oplus_{i\in I} \mathcal  R,\mathcal N), \, g\mapsto g_{|\oplus_{i\in I} \mathcal R}$$
is injective.
\end{lemma}

\begin{proof} Write $M=\oplus_{i\in I} R$. Then,
$$\Hom_{\mathcal R}(\prod_{i\in I} \mathcal  R,\mathcal N)\!=\!\Hom_{\mathcal R}(\mathcal  M^*,\mathcal N)\overset{\text{\ref{prop4}}}=\!M\otimes_R N=\underset{i\in I}\oplus N\subseteq \prod_{i\in I} N\!=\!\Hom_{\mathcal R}(\underset{i\in I}\oplus \mathcal  R,\mathcal N).$$

\end{proof}

\begin{proposition} \label{refpro} Let $\{\mathbb M_i\}_{i\in I}$ be a set of dual functors of $\mathcal R$-modules and let $N$ be an $R$-module. Then,
$$\Hom_{\mathcal R}(\prod_{i\in I} \mathbb  M_i,\mathcal N)=
\oplus_{i\in I} \Hom_{\mathcal R}(\mathbb  M_i,\mathcal N)$$
In particular, $(\prod_{i\in I}\mathbb  M_i)^*=\oplus_{i\in I}\mathbb  M_i^*$ and if
$\mathbb M_i$ is reflexive, for any $i$, then $\prod_{i\in I}\mathbb  M_i$ is reflexive.

\end{proposition}

\begin{proof} If
$f_{|\oplus_{i\in I}\mathbb  M_i}=0$, then $f=0$: Given
$m=(m_i)_{i\in I}\in \prod_{i\in I} \mathbb M_i$, define
$$g\colon \prod_{i \in I} \mathcal R\to \mathcal N,\, g((r_i)_i):=f((r_i\cdot m_i)_i).$$ Observe that $g_{|\oplus_i \mathcal R}=0$, hence $g=0$, by Lemma \ref{main}. Therefore, $f=0$.

Obviously,
$$
\oplus_{i\in I} \Hom_{\mathcal R}(\mathbb  M_i,\mathcal N)\subseteq
\Hom_{\mathcal R}(\prod_{i\in I} \mathbb  M_i,\mathcal N).$$

Let $f\in \Hom_{\mathcal R}(\prod_{i\in I} \mathbb  M_i,\mathcal N)$ and $J:=\{i\in I\colon f_i:=f_{|\mathbb  M_i}\neq 0\}$. For each $j\in J$, let $R_j$ be a commutative $R$-algebra and $m_j\in\mathbb M_j(R_j)$  such that $0\neq f_j(m_j)\in N\otimes_RR_j$. Let $S:=\prod_{j\in J}R_j$. The obvious morphism of $R$-algebras $S\to R_i$ is surjective, and this morphism of $R$-modules has a section. Hence, the natural morphism $\pi_i\colon \mathbb M_i(S)=\overline{\mathbb M_i}(S)\to
\overline{\mathbb M_i}(R_i)=\mathbb M_i(R_i)$ has a section of $R$-modules. Let $m'_i\in \mathbb M_i(S)$ be such that $\pi_i(m'_i)=m_i$.
The morphism of $\mathcal S$-modules $g\colon \prod_{J}\mathcal S\to \mathcal N\otimes_{\mathcal R} \mathcal S$, $g((s_j)):=f((s_j\cdot m'_j)_j)$ satisfies that
$g_{|\mathcal S}\neq 0 $, for every factor $\mathcal S\subset \prod_{J} \mathcal S$.
Then,  $\# J<\infty$ by Lemma \ref{main}.

Finally, define $h:=\sum_{j\in J} f_j\in \oplus_{i\in I} \Hom_{\mathcal R}(\mathbb  M_i,\mathcal N)$, then $f=h$.

\end{proof}

Let $\{\mathbb F_i\}_{i\in I}$ be a set of reflexive functors, then $\oplus_i\mathbb F_i$ is a reflexive functor and $\overline{\oplus_i\mathbb F_i}=\oplus_i \bar{\mathbb F_i}$.

\section{Quasi-coherent modules associated with finitely presented modules}

Let $M$ be an $R$-module. There exists an $R$-module $N$ such that $\mathcal M= \mathcal N^*$ iff 
$M$ is an $R$-module projective of finite type (see \cite{Amel2}). In other words, $\mathcal M =\mathcal M_{sch}$ iff $M$ is an $R$-module projective of finite type. $\mathcal M =\mathcal M_{sch}$ iff
$$M\otimes_R N'=\overline{\mathcal M}(N')=\overline{\mathcal M_{sch}}(N')= \Hom_R(M^*,N')$$
for any $R$-module $N'$.

\begin{theorem} The morphism ${\mathcal M^*}_{qc}\to \mathcal M^*$ is an epimorphism iff $M$ is a projective module of finite type.
\end{theorem}

\begin{proof} $\Rightarrow)$ The morphism $\overline{{\mathcal M^*}_{qc}}\to \overline{\mathcal M^*}$ is an epimorphism. Then,  the morphism
$$M^*\otimes_R N\to \Hom_R(M,N)$$
is surjective, for every $R$-module $N$.
Let $N=M$. Then, there exist $w_i\in M^*$ and $m_i\in M$, $i=1,\ldots, r$, such that $\sum_i w_i\otimes m_i\mapsto Id$. Therefore, $\sum_i w_i(m)m_i=m$, for every $m\in M$. Let $f\colon M\to R^r$, $f(m):=(w_i(m))$ and $g\colon R^r\to M$, $g(a_i):=\sum_i a_im_i$. Observe that $(g\circ f)(m)=g((w_i(m)))=\sum_i w_i(m)m_i=m$, that is, 
$g\circ f=Id$ and $M$ is a direct summand of $R^r$.

\end{proof}

\begin{corollary} A morphism $\mathcal N\to \mathcal M^*$ is an epimorphism iff $M$ is a projective module of finite type and the morphism $N\to M^*$ is an epimorphism.
\end{corollary}

\begin{proof} $\Rightarrow)$  $\mathcal N\to \mathcal M^*$  factors through the morphism
$ {\mathcal M^*}_{qc}\to  \mathcal M^*$, which is an epimorphism because $\mathcal N\to {\mathcal M^*}$ is an epimorphism. Then, $M$ is a projective module of finite type and the morphism $N\to M^*$ is an epimorphism.

$\Leftarrow)$ The morphism $\mathcal N\to {\mathcal M^*}_{qc}$ is an epimorphism because 
$ N\to  M^*$ is an epimorphism. The morphism
$ {\mathcal M^*}_{qc}\to  \mathcal M^*$ is an isomorphism because $M$ is a projective module of finite type. Then, 
$\mathcal N\to \mathcal M^*$ is an epimorphism .
\end{proof}

\begin{lemma} \label{1.1N} Let $f\colon \mathcal V_2\to \mathcal V_1$ be a morphism of $\mathcal R$-modules between quasi-coherent modules. Then, $f$ is an epimorphism iff $f^*\colon \mathcal V_1^*\to \mathcal V_2^*$ is a monomophism.
\end{lemma}

\begin{proof} $\Leftarrow)$ $\Coker f$ is the quasi-coherent module associated with to $\Coker f_R$, and $(\Coker f)^*=\Ker f^*=0$. Then, 
$\Coker f=(\Coker f)^{**}=0$.

\end{proof}

\begin{proposition} \label{1.2} If
$$0\to \mathcal V_1^*\overset i\to\mathcal V_2^*\overset\pi\to\mathcal M\to 0$$
is an exact sequence of morphisms of functors of $\mathcal R$-modules, then $M$ is a projective module of finite type.

\end{proposition}

\begin{proof} 1. $M$ is a finitely generated $R$-module, by Note \ref{2.12N}.

2. 
Given an $R$-module $N$, if we take $\Hom_{\mathcal R}(-,\mathcal N)$ on the above exact sequence we obtain the exact sequence
$$0\to \Hom_{\mathcal R}(\mathcal M,\mathcal N)\to V_2\otimes N\to V_1\otimes N\to 0$$
by Proposition \ref{tercer} and Lemma \ref{1.1N}.

3. Consider an exact sequence of morphisms of $R$-modules
$$0\to N_1\to N_2\to N_3\to 0$$
We obtain the diagram
$$\xymatrix{0 \ar[r] & \Hom_{\mathcal R}(\mathcal M,\mathcal N_1) \ar[r] \ar[d] &  V_2\otimes N_1 \ar[r] \ar[d]  & V_1\otimes N_1 \ar[r] \ar[d] & 0\\
0 \ar[r] & \Hom_{\mathcal R}(\mathcal M,\mathcal N_2) \ar[r] \ar[d] &  V_2\otimes N_2 \ar[r] \ar[d] & V_1\otimes N_2 \ar[r] \ar[d] & 0
\\ 0 \ar[r] & \Hom_{\mathcal R}(\mathcal M,\mathcal N_3) \ar[r]  &  V_2\otimes N_3 \ar[r] \ar[d]  & V_1\otimes N_3 \ar[r] \ar[d] & 0\\ &  & 0 & 0 & }$$
By the snake lemma, $\Hom_{\mathcal R}(\mathcal M,\mathcal N_2)\to \Hom_{\mathcal R}(\mathcal M,\mathcal N_3)$ is surjective. Then, $M$ is a projective module of finite type.
%2. Let $L=R^n\to M$ be an epimorphism of $R$-modules and $f\colon \mathcal L\to\mathcal M$ the associated morphism. Let $g\colon \mathcal N_2^*\to \mathcal L$ such that $f\circ g=\pi$.
%We can suppose that $g$ is an epimorphism, replacing $\mathcal N^*_2$ by $\mathcal N_2^*\oplus \mathcal L$ (and $\mathcal N_1^*$ by $\mathcal N_1^*\oplus \mathcal L$).
%Obviuosly, $g$ has  a section. Then, $\mathcal N_2^*=\mathcal L\oplus \Ker g$ and $\mathcal N_1^*=\Ker f\oplus \Ker g$. Hence, $\Ker f$ is a module scheme, because $\ekr g$ is a module scheme.
%
%The morphism $(\Ker f)_{qc}\to \Ker f$ is an epimorphism. By Proposition \ref{1.2}, $\Ker f$ is the quasicoherent module associated with a projective module of finite type. Hence, $M$ is flat and a  finitely presented module.

\end{proof}

\begin{proposition} Let $M$ be an $R$-module.
$M$ is a finitely presented $R$-module iff there exists an exact sequence of functors of $\mathcal R$-modules
$$\mathcal V_1^*\overset i\to\mathcal V_2^*\overset\pi\to\mathcal M\to 0.$$
\end{proposition}

\begin{proof} $\Leftarrow)$
1. $M$ is a finitely generated $R$-module, by Note \ref{2.12N}.

2. Let $f\colon L=R^n\to M$ be an epimorphism, $K:=\Ker f\subseteq L$,  and $\tilde f\colon \mathcal L\to\mathcal M$ and $i'\colon \mathcal K\to\mathcal L$ the associated morphisms. There exists a morphism $g\colon \mathcal V_2^*\to \mathcal L$ such that $ \tilde f\circ g=\pi$, because $\Hom_{\mathcal R}(\mathcal V_2^*,\mathcal L)\overset{\text{\ref{prop4}}}=V_2\otimes_R L\to V_2\otimes_R M\overset{\text{\ref{prop4}}}=\Hom_{\mathcal R}(\mathcal V_2^*,\mathcal M)$ is surjective.
We can suppose that $g$ is an epimorphism, replacing $\mathcal V_2^*$ by $\mathcal V_2^*\oplus\mathcal L$ (and $\mathcal V_1^*$ by $\mathcal V_1^*\oplus\mathcal L$). Consider the exact sequences of morphisms
$$K \to L\to  M\to 0$$
$$\xymatrix{\Hom_{\mathcal R}(\mathcal V_1^*,\mathcal K) & \Hom_{\mathcal R}(\mathcal V_1^*,\mathcal L)
& \Hom_{\mathcal R}(\mathcal V_1^*,\mathcal M) & \\  V_1\otimes_R K \ar@{=}[u]^-{\text{\ref{prop4}}} \ar[r]&
V_1\otimes_R L \ar@{=}[u]^-{\text{\ref{prop4}}} \ar[r] & V_1\otimes_R M \ar@{=}[u]^-{\text{\ref{prop4}}} \ar[r] & 0}$$
There exists a morphism $g'\colon \mathcal V_1^*\to \mathcal K$ such that $i'\circ g'=g\circ i$. The morphism $g'_R$ is surjective because $g_R$ is surjective. Then, $g'$ is an epimorphism and $K$ is a finitely generated module, by Note \ref{2.12N}. Hence, $M$ is a finitely presented module.

\end{proof}

\begin{proposition} \cite[Tag 058L]{stacks-project} \label{1.4} If $0\to \mathcal M_1\overset{i'}\to \mathcal M_2\overset{\pi'}\to\mathcal M_3\to 0$ is an exact sequence of functors of $\mathcal R$-modules and $M_3$ is a finitely presented module, then this exact sequence splits.
\end{proposition}

\begin{proof}
Let $\mathcal V_1^*\overset i\to\mathcal V_2^*\overset\pi\to\mathcal M_3\to 0$ be an exact sequence of $\mathcal R$-modules and let $\mathcal V_0^*:=\Ker i$ ($V_0:=\Coker[V_2\to V_1]$). Let $i_0\colon \mathcal V_0^*\to\mathcal V_1^*$ the inclusion morphism. Then,
$$0\to \mathcal V_0^*\overset{i_0} \to \mathcal V_1^*\overset i\to\mathcal V_2^*\overset\pi\to\mathcal M_3\to 0$$
is an exact sequence, and  
if we take $\Hom_{\mathcal R}(-,\mathcal N)$, for any quasi-coherent $\mathcal R$-module $\mathcal N$, the sequence
$$\xymatrix  @=12pt{0 \ar[r] & \Hom_{\mathcal R}(\mathcal M_3,\mathcal N) \ar[r] & \Hom_{\mathcal R}(\mathcal V_2^*,\mathcal N)\ar[r] &  \Hom_{\mathcal R}(\mathcal V_1^*,\mathcal N)\ar[r] &  \Hom_{\mathcal R}(\mathcal V_0^*,\mathcal N)\ar[r] & 0
\\ & (*) & V_2\otimes_R N \ar@{=}[u]^-{\text{\ref{prop4}}} &
V_1\otimes_R N \ar@{=}[u]^-{\text{\ref{prop4}}} & V_0\otimes_R N \ar@{=}[u]^-{\text{\ref{prop4}}}  & }$$
is exact. 

Let $f\colon \mathcal V_2^*\to\mathcal M_2$ be a morphism such that $\pi'\circ f=\pi$.  Let $g\colon \mathcal V_1^*\to \mathcal M_1$ be the  morphism such that the diagram
$$\xymatrix { & 0\ar[r] & \mathcal M_1 \ar[r]^-{i'} &
\mathcal M_2 \ar[r]^-{\pi'} & \mathcal M_3 \ar[r] & 0\\
0 \ar[r] & \mathcal V_0^* \ar[r]^-{i_0} & \mathcal V_1^* \ar[r]^-i \ar[u]^-g& \mathcal V_2^* \ar[r]^-\pi \ar[u]^-f & \mathcal M_3 \ar[r] \ar@{=}[u] & 0}$$
is commutative.
Observe that $i'\circ g\circ i_0=f\circ i\circ i_0=0$, then $g\circ i_0=0$. 
 By the exact sequence $(*)$, there exists a morphism $f'\colon \mathcal V_2^*\to \mathcal M_1$ such that $g=f'\circ i$. Therefore, $f-i'\circ f'$ is zero over $\mathcal V_1^*$. Then, there exists a morphism $s\colon \mathcal M_3\to \mathcal M_2$ such that $f-i'\circ f'=s\circ \pi$. Then, $\pi=\pi'\circ f=
\pi'\circ (f-i'\circ f')=\pi'\circ s\circ \pi$ and $\pi'\circ s=Id$.

\end{proof}

\begin{proposition} \label{3.6} Let $P$ be a finitely presented $R$-module,  $M\to P$ an epimorphism and $f\colon \mathcal M\to \mathcal P$ the associated morphism. Consider the exact sequence of morphisms of $\mathcal R$-modules
$$0\to \Ker f\to \mathcal M\to \mathcal P\to 0$$
Then, the sequence of morphisms  of $\mathcal R$-modules
$$0\to \mathcal P^*\to \mathcal M^*\to (\Ker f)^*\to 0$$ is exact.
More generally, the sequence  of morphisms of $R$-modules
$$0\to \Hom_{\mathcal R}(\mathcal P,\mathcal N)\to \Hom_{\mathcal R}(\mathcal M,\mathcal N)\to \Hom_{\mathcal R}(\Ker f,\mathcal N)\to 0$$
is exact, for any $R$-module $N$.

Finally, $\Ker f$ is a reflexive functor of $\mathcal R$-modules.
\end{proposition}

\begin{proof} We have to prove that every morphism $g\colon \Ker f\to \mathcal N$  lifts to a morphism
$\mathcal M\to \mathcal N$. It is equivalent to prove that the exact sequence of morphisms of $\mathcal R$-modules
$$0\to \mathcal N\to \mathcal N\underset{\Ker f}\oplus \mathcal M\to \mathcal P\to 0$$
splits. $\mathcal N\underset{\Ker f}\oplus \mathcal M$ is the  quasicoherent $\mathcal R$-module associated with the cokernel of the morphism $\Ker f_R\to N\oplus M,$ $k\mapsto (k,-g_R(k))$. By \ref{1.4}, this exact sequence splits.

Finally, the sequence of morphisms of $\mathcal R$-modules
$$0\to (\Ker f)^{**} \to \mathcal M^{**}=\mathcal M\to \mathcal P^{**}=\mathcal P$$
is exact, then $(\Ker f)^{**}=\Ker f$.

\end{proof}

\section{FP-functors}

\begin{definition} We shall say that a functor of $R$-modules $\mathbb M$ is an FP-functor if $$\Hom_{\mathcal R}(\mathbb M, \ilim{i}\mathcal N_i)=\ilim{i}\Hom_{\mathcal R}(\mathbb M, \mathcal N_i)$$
for every direct system of quasi-coherent modules $\{\mathcal N_i\}$.
\end{definition}

\begin{example} Module schemes are FP-functors:
$$\Hom_{\mathcal R}(\mathcal M^*, \ilim{i}\mathcal N_i)\overset{\text{\ref{prop4}}}=
M\otimes_R (\ilim{i} N_i)=\ilim{i} (M\otimes_R N_i)\overset{\text{\ref{prop4}}}=
\ilim{i}\Hom_{\mathcal R}(\mathcal M^*, \mathcal N_i).$$

\end{example}

\begin{theorem} \label{4.8} Let $\mathbb F_1$ and  $\mathbb F_2$ 
be FP-functors of $\mathcal R$-modules and $f\colon \mathbb F_1\to\mathbb F_2$ a morphism of $\mathcal R$-modules.
Then, $\Coker f$ is an FP-functor of $\mathcal R$-modules.\end{theorem}

\begin{proposition} \label{1.9} Let $P$ be a finitely presented module and $\{\mathbb M_i\}$ a direct system of $\mathcal R$-modules. Then,
$$\Hom_{\mathcal R}(\mathcal P,\ilim{i} \mathbb M_i)=\ilim{i} 
\Hom_{\mathcal R}(\mathcal P,\mathbb M_i).$$
In particular, $\mathcal P$ is an FP-functor.\end{proposition}

\begin{proof} By \ref{tercer},
$\Hom_{\mathcal R}(\mathcal P,\ilim{i} \mathbb M_i)=
\Hom_{R}( P,\ilim{i} \mathbb M_i(R))=\ilim{i} \Hom_{R}( P,\mathbb M_i(R)) $ $=$ $\ilim{i} 
\Hom_{\mathcal R}(\mathcal P,\mathbb M_i)$.
\end{proof}

\begin{proposition} \label{4.4} Let $M$ be an $R$-module. $M$ is a finitely presented module iff $\mathcal M$ is an FP-functor of $\mathcal R$-modules.

\end{proposition}

\begin{proof} $\Leftarrow)$ Any $R$-module is a direct limit of
finitely presented modules. Write $M=\ilim{i} P_i$, where $P_i$ is a finitely presented module, for any $i$. Then, $Id\colon M\to M$ factors through a morphism $f_i\colon M\to P_i$,  for some $i$. Then, $M$ is a direct summand of $P_i$, and it is finitely presented.

$\Rightarrow)$ It is an immediate  consequence of Proposition \ref{1.9}.

\end{proof}

Recall Notation \ref{notation}.

\begin{proposition}  \label{4.7} $\mathbb M$ is an FP-functor iff $\overline{\mathbb M^*}$ commutes with direct limits of $R$-modules.
\end{proposition}

\begin{proof} It is an immediate consequence  of the equality $$\Hom_{\mathcal R}(\mathbb M,\ilim{i}\mathcal N_i)=\overline{\mathbb M^*}(\ilim{i} N_i).$$

\end{proof}

\begin{proposition} \label{4.75} Let $M$ be an $R$-module and let  $N \subseteq M$ be an $R$-submodule.  Let $\tilde N$ be the image of the obvious morphism $\mathcal N\to \mathcal M$. $\tilde N$ is an FP-functor iff there exist finitely presented modules $P$ and $P'$ and an exact sequence of morphisms
of $\mathcal R$-modules
$$0\to \tilde N\to \mathcal P\to \mathcal P'\to 0$$
(and in particular, $N$ is isomorphic to a finitely generated submodule of $P$).

\end{proposition}

\begin{proof} $\Rightarrow)$ Write $M=\ilim{i} P_i$, where $\{P_i\}$ is a direct system of finitely presented $R$-modules. Then, $\mathcal M=\ilim{i}\mathcal P_i$ and for some $i$ the morphism $\tilde N\hookrightarrow \mathcal M$ factors through an injective morphism $\tilde N\to \mathcal P_i$.
Consider the composite morphism $\mathcal N\to \tilde N \to\mathcal P_i$.
$\Coker[\tilde N\to \mathcal P_i]=\Coker[\mathcal N\to \mathcal P_i]=:\mathcal Q$ is quasi-coherent. Then we have the exact sequence of morphisms
$$0\to \tilde N\to \mathcal P_i\to \mathcal Q\to 0$$
By \ref{4.8}, $\mathcal Q$ is an FP-functor. By \ref{4.4}, $Q$ is a finitely presented $R$-module. We are done.

$\Leftarrow)$  By \ref{3.6}, the sequence of morphisms of $R$-modules
$$0\to \mathcal P'^*\to \mathcal P^* \to (\Ker f)^*\to 0$$
is exact. $\overline{\mathcal P'^*}$ and  $\overline{\mathcal P^*} $ commute with direct limits of $\mathcal R$-modules, therefore $\overline{(\Ker f)^*}$ commutes with direct limits of $\mathcal R$-modules. Hence, $\Ker f$ is an FP-functor by Proposition \ref{4.7}.

\end{proof}

\begin{proposition} Let $P$ be a finitely presented module and $f\colon \mathcal M\to \mathcal P$ an epimorphism. Then, $\Ker f$ is an FP-functor iff $M$ is a finitely presented module.\end{proposition}

\begin{proof} $\Rightarrow)$ By Proposition \ref{3.6}, the rows of the diagram
$$\xymatrix @C12pt {0 \ar[r] & \Hom_{\mathcal R}(\mathcal P,\ilim{i} \mathcal N_i) \ar[r] & \Hom_{\mathcal R}(\mathcal M,\ilim{i} \mathcal N_i) \ar[r] & \Hom_{\mathcal R}(\Ker f,\ilim{i} \mathcal N_i) \ar[r] & 0\\ 0 \ar[r] & \ilim{i} \Hom_{\mathcal R}(\mathcal P,\mathcal N_i) \ar@{=}[u]  \ar[r] & \ilim{i} \Hom_{\mathcal R}(\mathcal M,\mathcal N_i)  \ar[r] & \ilim{i} \Hom_{\mathcal R}(\Ker f,\mathcal N_i) \ar@{=}[u]  \ar[r] &  0}$$
are exact, then $\Hom_{\mathcal R}(\mathcal M,\ilim{i} \mathcal N_i)=\ilim{i} \Hom_{\mathcal R}(\mathcal M,\mathcal N_i)$.
Then, $\mathcal M$ is an FP-functor and $M$ is a finitely presented $R$-module, by Proposition \ref{4.4}.

$\Leftarrow)$ It is an immediate consequence of Proposition \ref{4.75}.
\end{proof}

\begin{theorem}  \label{4.6N} $\mathbb M$ is an FP-functor of $\mathcal R$-modules iff $\mathbb M^*$ is the cokernel of a morphism $F\colon \oplus_{i} \mathcal P_i^*\to \oplus_{j} \mathcal Q_j^*$, where $P_i,Q_j$ are finitely presented $R$-modules, for every $i,j$. 
\end{theorem}

\begin{proof} $\Leftarrow)$ $\overline{\mathcal P^*}(\ilim{i} N_i)=\Hom_R(P,\ilim{i} N_i) =\ilim{i} \Hom_R(P,\ilim{i} N_i)=\ilim{i} \overline{\mathcal P^*}(N_i)$, for every finitely presented $R$-module $P$. Then, $\overline{\oplus_{i} \mathcal P_i^*}=\oplus_{i} \overline{\mathcal P_i^*}$ and $\overline{\oplus_{j} \mathcal Q_j^*}=\oplus_{j} \overline{\mathcal Q_j^*}$
commute with direct limits. Hence, $\overline{\mathbb M^*}=\overline{\Coker F}$ commutes with direct limits and $\mathbb M$ is a FP-functor, by \ref{4.7}.

$\Rightarrow)$ Choose a set $A$ of representatives of the isomorphism classes of finitely presented $R$-modules. Let $B$ be the set of the pairs
$(\mathcal P^*,g)$, where $P\in A$ and $g\in \Hom_{\mathcal R}(\mathcal P^*,\mathbb M^*)$. The obvious morphism
$$G\colon \oplus_{(\mathcal P^*,g)\in B}\mathcal P^*\to \mathbb M^*$$
is an epimorphism: Let $Q$ be a finitely presented module and
$g'\in \overline{\mathbb M^*}(Q)=\Hom_{\mathcal R}(\mathcal Q^*,\mathbb M^*)$. Obviously, through the morphism 
$g'\colon \overline{\mathcal Q^*}\to \overline{\mathbb M^*}$, $Id_Q$ is mapped to $g'$. Hence,  the morphism $$ \overline{\oplus_{(\mathcal P^*,g)\in B}\mathcal P^*}(Q)\to 
\overline{\mathbb M^*}(Q)$$ is surjective.  Every module  is a direct limit of finitely presented $R$-modules. $\overline{\oplus_{(\mathcal P^*,g)\in B}\mathcal P^*}$ and $\overline{\mathbb M^*}$ commute with direct limits, then
$G$ is an epimorphism.

$\Ker G$ is a dual functor: By Proposition \ref{trivial}, $G$ is the dual morphism of a morphism $H\colon \mathbb M\to \prod_{(\mathcal P^*,g)\in B}\mathcal P$. Then, $\Ker G=(\Coker H)^*$. 
$\overline{\Ker G}$ commutes with direct limits, hence again there exists an epimorphism $\oplus_i \mathcal Q_i^*\to \Ker G$. We are done.

\end{proof}

\begin{corollary} Let $K$ be a field and $\mathbb M$ an  $\mathcal K$-module.  $\mathbb M$ is an FP-functor iff $\mathbb M^*$ is quasi-coherent.

\end{corollary}

Let $K$ be a field.
In \cite[2.2]{Amel}, it has been proved that reflexive FP-functors of $\mathcal K$-modules are module schemes.

\begin{proposition} If $\mathbb P$ is an FP-functor of $\mathcal R$-modules, then $\mathbb P_{|S}$ is an FP-functor of $\mathcal S$-modules, for any commuative $R$-algebra.\end{proposition}

\begin{proof} Let $\{\mathcal N_i\}$ be a direct system of $\mathcal S$-modules. Then,

$$\Hom_{\mathcal S}(\mathbb P_{|S},\ilim{i} \mathcal N_i)\overset{\text{\ref{adj2}}}=\! \Hom_{\mathcal R}(\mathbb P,\ilim{i} \mathcal N_i)\!=\!\ilim{i}\Hom_{\mathcal R}(\mathbb P, \mathcal N_i)
\overset{\text{\ref{adj2}}}=\!\ilim{i} \Hom_{\mathcal S}(\mathbb P_{|S},\mathcal N_i).$$

\end{proof}

\begin{corollary} Let $\mathbb P$ be an FP-functor. Then,
$$\mathbb Hom_{\mathcal R}(\mathbb P, \ilim{i}\mathcal N_i)=\ilim{i}\mathbb Hom_{\mathcal R}(\mathbb P, \mathcal N_i).$$
\end{corollary}

\begin{lemma} \label{4.5} Let $P$ and $P'$ be finitely presented modules and
$f\colon \mathcal P\to \mathcal P'$ an epimorphism. Then, 

\begin{enumerate} 

% \item $(\Ker f)^*$ is an FP-functor of $\mathcal R$-modules.

%\item Let $N=\prod_{i} N_i$, $\Hom_{\mathcal R}(\Ker f, \mathcal N)=\prod_{i} \Hom_{\mathcal R}(\Ker f,\mathcal N_i)$.
%
%\item Let $N=\prod_{i} N_i$, $\Hom_{\mathcal R}((\Ker f)^*, \mathcal N)=\prod_{i} \Hom_{\mathcal R}((\Ker f)^*,\mathcal N_i)$.

\item $\Hom_{\mathcal R}(\oplus_i \mathcal N_i^*,\Ker f)= \Hom_{\mathcal R}((\oplus_i \mathcal N_i^*)_{sch},\Ker f)$.

\item $\Ker f_R$ is a finitely generated $R$-module and the morphism $(\Ker f)_{qc}\to \Ker f$ is an epimorphism. Let $L$ be a finite free module and let $\pi\colon \mathcal L\to \Ker f$ an epimorphism.
There exists an exact sequence of morphisms of $\mathcal R$-modules
$$\mathcal V^*\to\mathcal L\overset\pi\to \Ker f\to 0$$
%\item  $\Hom_{\mathcal R}(\oplus_i \mathcal N_i^*,(\Ker f)^*)= \Hom_{\mathcal R}((\oplus_i \mathcal N_i^*)_{sch},(\Ker f)^*)$.

\end{enumerate}

\end{lemma}

\begin{proof} (1) Given a finitely presented $R$-module, $Q$, we have that
$$\aligned \Hom_{\mathcal R}(\oplus_i \mathcal N_i^*,\mathcal Q) & =\prod_i\Hom_{\mathcal R}(\mathcal N_i^*,\mathcal Q)\overset{\text{\ref{prop4}}}=
\prod_i (N_i\otimes_R Q)=(\prod_i N_i)\otimes_R Q\\ & \overset{\text{\ref{1211}}}=
\Hom_{\mathcal R}((\oplus_i \mathcal N_i^*)_{sch},\mathcal Q).\endaligned$$

Write $\mathbb N= \oplus_i \mathcal N_i^*$. From the commutative diagram of exact rows
$$\xymatrix @=10pt{0\ar[r] & \Hom_{\mathcal R}(  \mathbb N,\Ker f)\ar[r] & \Hom_{\mathcal R}(\mathbb N,\mathcal P)\ar[r] & \Hom_{\mathcal R}(\mathbb N,\mathcal P')
\\ 0\ar[r] &  \Hom_{\mathcal R}(\mathbb N_{sch},\Ker f)\ar[r]  & \Hom_{\mathcal R}(\mathbb N_{sch},\mathcal P)\ar[r] \ar@{=}[u] &  \Hom_{\mathcal R}(\mathbb N_{sch},\mathcal P')\ar@{=}[u] }$$
we obtain that $\Hom_{\mathcal R}(\oplus_i \mathcal N_i^*,\Ker f)= \Hom_{\mathcal R}((\oplus_i \mathcal N_i^*)_{sch},\Ker f)$.

(2)  By \ref{3.6}, the sequence of morphisms of $R$-modules
$$0\to \mathcal P'^*\to \mathcal P^* \to (\Ker f)^*\to 0$$
is exact. By \ref{4.8}, $ (\Ker f)^*$ is an FP-functor.

$\Ker f_R$ is a finitely generated $R$-module and the morphism $(\Ker f)_{qc}\to \Ker f$ is an epimorfism. Hence, there exist a finite free module $L$ and an epimorphism $\pi\colon \mathcal L\to \Ker f$. $\Ker \pi = (\Coker \pi^*)^*$ and by \ref{4.8},
$\Coker \pi^*$ is an $FP$-functor. By \ref{4.6N}, there exists an epimorphism
$g'\colon \oplus \mathcal Q_i^*\to \Ker \pi$. Observe that
$$\Hom_{\mathcal R}(\mathbb N, \mathcal L)=\Hom_{(\mathcal R}(\mathbb N_{sch}, \mathcal L)$$
for any $\mathcal R$-module $\mathbb N$, and $\Hom_{\mathcal R}(\oplus \mathcal Q_i^*, \Ker f)=
\Hom_{\mathcal R}((\oplus \mathcal Q_i^*)_{sch}, \Ker f)$, then
$$\Hom_{\mathcal R}(\oplus \mathcal Q_i^*, \Ker \pi)=\Hom_{\mathcal R}((\oplus \mathcal Q_i^*)_{sch}, \Ker \pi)$$
Then, $g'$ factors through an epimorphism $\mathcal N^*:=(\oplus \mathcal Q_i^*)_{sch}\to \Ker \pi$
and we have the obvious exact sequence
$$\mathcal N^*\to \mathcal L\to \Ker f\to 0$$

%From the commutative diagram of exact rows
%$$\xymatrix @=12pt{0\ar[r] & \Hom_{\mathcal R}(\mathcal P',\mathcal N)\ar[r] & \Hom_{\mathcal R}(\mathcal P,\mathcal N)\ar[r] & \Hom_{\mathcal R}(\Ker f,\mathcal N)\ar[r] & 0
%\\ 0\ar[r] &  \prod_{i} \Hom_{\mathcal R}(\mathcal P',\mathcal N_i) \ar@{=}[u] \ar[r]  & \prod_i\Hom_{\mathcal R}(\mathcal P,\mathcal N_i)\ar[r] \ar@{=}[u] &\prod_{i} \Hom_{\mathcal R}(\Ker f,\mathcal N_i) \ar[r] & 0}$$
%we obtain that $\Hom_{\mathcal R}((\Ker f)^*,\mathcal N)=\prod_i \Hom_{\mathcal R}((\Ker f)^*,\mathcal N_i)$.
%
%
%(4) Given a finitely presented $R$-module, $Q$, we have that
%$$\Hom_{\mathcal R}(\mathcal Q^*,\mathcal N)=Q\otimes_R(\prod_i N_i)=\prod_i (Q\otimes_R N_i)=\prod_i
%\Hom_{\mathcal R}(\mathcal Q^*,\mathcal N_i).$$
%
%From the commutative diagram of exact rows
%$$\xymatrix @=12pt{0\ar[r] & \Hom_{\mathcal R}( (\Ker f)^*,\mathcal N)\ar[r] & \Hom_{\mathcal R}(\mathcal P^*,\mathcal N)\ar[r] & \Hom_{\mathcal R}(\mathcal P'^*,\mathcal N)
%\\ 0\ar[r] & \prod_{i} \Hom_{\mathcal R}((\Ker f)^*,\mathcal N_i)\ar[r]  & \prod_i\Hom_{\mathcal R}(\mathcal P^*,\mathcal N_i)\ar[r] \ar@{=}[u] & \prod_{i} \Hom_{\mathcal R}(\mathcal P'^*,\mathcal N_i)\ar@{=}[u] }$$
%we obtain that $\Hom_{\mathcal R}((\Ker f)^*,\mathcal N)=\prod_i \Hom_{\mathcal R}((\Ker f)^*,\mathcal N_i)$.

\end{proof}

\section{Mittag-Leffler modules}

\begin{definition} \cite[Tag 0599]{stacks-project} Let $M$ an $R$-module. $M$ is said to be a Mittag-Leffler module  if for every finite free $R$-module $F$ and morphism of $R$-modules $f\colon F\to M$, there exists a finitely presented $R$-module $Q$ and a morphism of $R$-modules $g\colon F\to Q$ such that $\Ker[F\otimes_R N\to M\otimes_R N]=\Ker[F\otimes_R N\to Q\otimes_R N]$ for every $R$-module $N$, that is,
the kernel of the associated morphism $\tilde f\colon \mathcal F\to \mathcal M$ is equal to the kernel of the associated morphism $\tilde g\colon \mathcal F\to \mathcal Q$, that is, we have a commutative diagram
$$\xymatrix{& & \mathcal M\\ \mathcal F  \ar[rru]^-{\tilde f} \ar[r] \ar[rrd]_-{\tilde g}  & \Ima \tilde f=\Ima\tilde g \ar@{^{(}->}[ru]  \ar@{^{(}->}[rd] & \\ & & \mathcal Q}$$

\end{definition}

\begin{theorem} \label{4.6} Let $M$ be an $R$-module. $M$ is a Mittag-Leffler module iff for every finitely generated $R$-submodule 
$N\subseteq M$ the image of the morphism $\mathcal N\to \mathcal M$ is an FP-functor of $\mathcal R$-modules.

\end{theorem}

\begin{proof} $\Rightarrow)$ Let $N\subseteq M$ be a finitely
generated submodule. Let $F$ be a  finite free module
and  an epimorphism $F\to N$. Let $f$ be the composite morphism
$F\to N\to M$. Let $Q$ be a finitely presented module and $g\colon F\to Q$ a morphism of $R$-modules, such that the images of the associated morphisms $\tilde f\colon \mathcal F\to \mathcal M$ and $\tilde g\colon \mathcal F\to \mathcal Q$ are equal. Observe that $\Coker\tilde g$ is equal to the quasi-coherent functor of modules associated with 
$\Coker g$, which is a finitely presented module. By \ref{4.75}, 
$ \Ima\tilde g$ is an FP-functor. Then, $$\Ima[\mathcal N\to\mathcal M]=\Ima \tilde f=\Ima \tilde g$$ is an FP-functor.

$\Leftarrow)$ Let $F$ be a finite free module, $f\colon F\to M$  a morphism of $R$-modules and $N:=\Ima f\subseteq M$. Consider the associated morphism
$\tilde f\colon \mathcal F\to \mathcal M$. Obviously $\Ima\tilde f$ is equal to the image of the morphism $\mathcal N\to\mathcal M$. Then, $\Ima \tilde f$ is an FP-functor.

Write $M=\ilim{i} P_i$, where $P_i$ is a finitely presented $R$-module, for every $i$. Then,
$$\Hom_{\mathcal R}(\Ima \tilde f, \mathcal M)=\Hom_{\mathcal R}(\Ima \tilde f,\ilim{i} \mathcal P_i)=\ilim{i}\Hom_{\mathcal R}(\Ima \tilde f, \mathcal P_i)$$
Hence, there exist an $i$ and a morphism $\Ima \tilde f\to \mathcal P_i$ such that the  composite morphism $\Ima \tilde f\to \mathcal P_i\to \mathcal M$ is the inclusion morphism. Hence, the morphism $\Ima \tilde f\to \mathcal P_i$ is a monomorphism and we have the commutative diagram 
$$\xymatrix{& & \mathcal M\\ \mathcal F  \ar[rru]^-{\tilde f} \ar[r] \ar[rrd]  & \Ima \tilde f \ar@{^{(}->}[ru]  \ar@{^{(}->}[rd] & \\ & & \mathcal P_i}$$
Hence, $M$ is a Mittag-Leffler module.

\end{proof}

\begin{theorem} Let $M$ be an $R$-module. $M$ is a Mittag-Leffler module iff $\mathcal M$ is a direct limit of FP-functors of $\mathcal R$-submodules.\end{theorem}

\begin{proof} $\Leftarrow)$ Let $\mathcal M$ be the direct limit
of a direct system of FP-functors of submodules $\{\mathbb P_j\} $. 
Let $L=R^n$ be a finite free module. A morphism $f\colon \mathcal L\to \mathcal M=\ilim{i}\mathbb P_i$, factors through a morphism
$g\colon \mathcal L\to \mathbb P_j$, by Proposition \ref{1.9}.
Any $R$-module is a direct limit of finitely presented modules, write $M=\ilim{i} P_i$, where $P_i$ is a finitely presented $R$-module, for every  $i$. Then, each canonical morphism $\mathbb P_j\to \mathcal M=\ilim{i}\mathcal P_i$, factors through a morphism
$\mathbb P_j\to \mathcal P_i$ for some $i$. Hence, we have a commutative diagram
$$\xymatrix{\mathcal L \ar[r] & \Ima f \ar@{^{(}->}[r] \ar@{^{(}->}[rd] & \mathcal M & \\ &  & \mathbb P_j \ar[r] \ar@{^{(}->}[u] & \mathcal P_i \ar[ul]}$$
Then, $M$ is a Mittag-Leffler module.

$\Rightarrow)$ Let $\{M_i\}$ be the set of all the finitely generated submodules of $M$. Let ${\tilde M_i}$ the image of the morphism $\mathcal M_i\to\mathcal M$. By Theorem \ref{4.6}, ${\tilde M_i}$ is an FP-functor for any $i$. We have the morphisms $\mathcal M_i \to {\tilde M_i}\hookrightarrow \mathcal M$. Taking direct limits we have
$$\xymatrix{\mathcal M= \ilim{i} {\mathcal M_i}  \ar[r] \ar@/^4mm/[rr]^-{Id} & \ilim{i} {\tilde M_i} \ar@{^{(}->}[r] & \mathcal M}$$
Then, $ \ilim{i} {\tilde M_i}=\mathcal M$.

\end{proof}

%\begin{lemma} Let $f\colon \mathcal N\to \mathcal M$ be a morphism of $\mathcal R$-modules. Choose a set $B$ of representatives of the isomorphism classes of finitely presented $R$-modules. Let $A$ be the set of pairs $(\mathcal Q^*,g)$ where $Q\in B$ and $g\in \Hom_{\mathcal R}(\mathcal Q^*,\mathcal N)$ such that $f\circ g=0$. The cokernel of the obvious morphism
%$$f'\colon\oplus_{(\mathcal Q^*,g)\in A} \mathcal Q^*\to \mathcal N$$
%is isomorphic to $\Ima f$.
%
%\end{lemma}
%
%\begin{proof}  Obviously, $\Ima f'\subseteq \Ker f$, then $\Coker f'\to \Ima f$ is an epimorphism.
%We only have to prove  that $(\Coker f')(S)\to (\Ima f)(S)$ is injective for any finitely presented $R$-algebra $S$.
%
%\end{proof}

\begin{corollary} \cite[Tag 059H]{stacks-project} \label{CML} Let $M$ be an $R$-module. $M$ is a Mittag-Leffler module iff the natural morphism $M\otimes_R\prod_{i\in I} Q_i\to \prod_{i\in I} (M\otimes_R Q_i)$ is injective, for every set of $R$-modules $\{Q_i\}_{i\in I}$.

\end{corollary}

\begin{proof} $\Rightarrow)$ Let $\{M_j\}$ be the set of the submodules of $M$ and $\bar {\mathcal M_j}$ the image of the morphism $\mathcal M_j\to\mathcal M$. Then, $\mathcal M=\ilim{j}{\tilde M_j}$. Then,
$$\aligned  & M\otimes_R \prod_i Q_i  \overset{\text{\ref{1211}}}=\Hom_{\mathcal R}((\oplus_i \mathcal Q_i^*)_{sch},\mathcal M)\overset{\text{\ref{L5.11}}}=\ilim{j} \Hom_{\mathcal R}((\oplus_i \mathcal Q_i^*)_{sch},{\tilde M_j})\\ & \overset{\text{\ref{4.5}}}=\ilim{j} \Hom_{\mathcal R}(\oplus_i \mathcal Q_i^*,{\tilde M_j})=\ilim{j} \prod_i\Hom_{\mathcal R}(\mathcal Q_i^*,{\tilde M_j})\hookrightarrow
\prod_i\ilim{j}\Hom_{\mathcal R}(\mathcal Q_i^*,{\tilde M_j})
\\ & \overset{\text{\ref{L5.11}}}=\prod_i \Hom_{\mathcal R}(\mathcal Q_i^*,\mathcal M)=\prod_i(M\otimes_R Q_i)\endaligned$$

%Observe that
% $$M\otimes_RN=\Hom_{\mathcal R}(\mathcal N^*,\mathcal M)\overset{\text{\ref{L5.11}}}=\ilim{j} \Hom_{\mathcal R}( \mathcal N^*, \bar {\mathcal M_i})=\ilim{j} \Hom_{\mathcal R}(\bar {\mathcal M_i}^*, \mathcal N),$$  for any $R$-module $N$, and the morphisms $\Hom_{\mathcal R}( \mathcal N^*, \bar {\mathcal M_j})\to \Hom_{\mathcal R}( \mathcal N^*, \bar {\mathcal M_k})$ are injective, for every $k>j$.
%Let $N=\prod_{i} Q_i$.   Then,
%$$\aligned M&\otimes_R \prod_{i} Q_i =\ilim{j} \Hom_R(\bar {\mathcal M_j}^*,\mathcal N)\overset{\text{\ref{4.5}}}=
%\ilim{j} \prod_i\Hom_R( \bar {\mathcal M_j}^*,\mathcal Q_i)\\ &\hookrightarrow
%\prod_i\ilim{j} \Hom_R(\bar {\mathcal M_j}^*,\mathcal Q_i)=
%\prod_i (M\otimes Q_i)\endaligned$$
%
$\Leftarrow)$ Let $L$ be a finitely generated $R$-module and $L\to M$ a morphism of $R$-modules. By Theorem \ref{4.6}, we have to prove that the image of the associated morphism $f\colon \mathcal L\to \mathcal M$ is an FP-functor. Let $L'$ be a finite free $R$-module and
$L'\to L$ an epimorphism. Obviously, the image of the composite morphism $\mathcal L'\to \mathcal L\to \mathcal M$ is equal to the image of the morphism $\mathcal L\to \mathcal M$. Then, we can suppose that $L$ is a finite free module.

Consider the dual morphism $f^*\colon \mathcal M^*\to \mathcal L^*$. $\Coker f^*$ is an FP-functor, by \ref{4.8},  and $\Ker f=(\Coker f^*)^*$. 
By Theorem \ref{4.6N}, there exists an epimorphism 
$\pi\colon \oplus_i\mathcal Q_i^*\to \Ker f$.
Consider the commutative diagram
$$\xymatrix{\oplus_i\mathcal Q_i^* \ar[r] \ar[dr] & \mathcal L \ar[r] & \Ima f \ar@{^{(}->}[r] & \mathcal M\\ & (\oplus_i\mathcal Q_i^* )_{sch} \ar[u] & & }$$

$\Hom_{\mathcal R}((\oplus_i \mathcal Q_i^*)_{sch},\mathcal M) \overset{\text{\ref{1211}}}=M\otimes_R \prod_i Q_i\hookrightarrow\prod_i(M\otimes_R Q_i)=\Hom_{\mathcal R}(\oplus_i\mathcal Q_i^*,\mathcal M)$. Then, the morphism
$(\oplus_i \mathcal Q_i^*)_{sch}\to \mathcal M$ is zero and
$$\Coker[(\oplus_i \mathcal Q_i^*)_{sch}\to \mathcal L]=\Ima f$$
$(\oplus_i \mathcal Q_i^*)_{sch}$ and $\mathcal L$ are FP-functors, then $\Ima f$ is an FP-functor by \ref{4.8} .

\end{proof}

\begin{lemma} \label{epi} Let $M$ be a Mittag-Leffler module, and $N_1\subseteq N_2\subseteq M$ two finitely generated submodules. Let ${\tilde N_1} \subseteq {\tilde N_2}$ be the images of the morphisms  $\mathcal N_1,\mathcal N_2\to \mathcal M$. The dual morphism  
${\tilde N_2}^*\to {\tilde N_1}^*$ is an epimorphism.
\end{lemma}

\begin{proof} $M$ is equal to a direct limit of finitely presented modules, $M=\ilim{i} P_i$. The morphism ${\tilde N_2}\hookrightarrow \mathcal M$ factors through a morphism ${\tilde N_2}\to \mathcal P_i$, which is a monomorphism. 
$\Coker[{\tilde N_2}\to \mathcal P_i]=\Coker[{\mathcal N_2}\to \mathcal P_i]$, which is the quasi-coherent module associated with $P'=\Coker[{N_2}\to P_i]$. We have the exact sequence
$$0\to {\tilde N_2}\to \mathcal P_i\to \mathcal P'\to 0$$
By Proposition \ref{3.6}, the morphism $\mathcal P_i^*\to  {\tilde N_2}^*$ is an epimorphism. Consider the composite morphism
${\tilde N_1}\hookrightarrow {\tilde N_2}\hookrightarrow \mathcal P_i$. $\Coker[{\tilde N_1}\to \mathcal P_i]=\Coker[{\mathcal N_1}\to \mathcal P_i]$, which is the quasi-coherent module associated with $Q'=\Coker[{N_1}\to P_i]$. Again by Proposition \ref{3.6}, the morphism $\mathcal P_i^*\to  {\tilde N_1}^*$ is an epimorphism. Then,  the morphism  
${\tilde N_2}^*\to {\tilde N_1}^*$ is an epimorphism.
\end{proof}

Let  $\{F_i\}$ be a direct system of finitely presented modules, so that $M=\ilim{i} F_i$. $M$ is a strict Mittag-Leffler module iff
for every commutative $R$-algebra $S$ and index $i$ there exists 
$j\geq i$ such that $\Ima[\mathcal M^*(S)\to \mathcal F_i^*(S)]=
\Ima[\mathcal F_j^*(S)\to \mathcal F_i^*(S)]$ (see \cite[II 2.3.2]{RaynaudG}).
 
\begin{theorem} \label{5.6}  Let $M$ be an $R$-module. $M$ is a strict Mittag-Leffler module iff for every finitely generated submodule $N\subseteq M$ the image ${\tilde N}$ of the associated morphism
$\mathcal N\to \mathcal M$ is an FP-functor and the morphism
$\mathcal M^*\to {\tilde N}^*$ is an epimorphism.
\end{theorem}

\begin{proof} Let  $\{F_i\}$ be a direct system of finitely presented $R$-modules, so that $M=\ilim{i} F_i$. Let $\{N_k\}$ be the set of the finitely generated submodules of $M$, and let ${\tilde N_k}$ be the image of the morphism $\mathcal N_k\to \mathcal M$. Then, $\mathcal M=\ilim{k} {\tilde N_k}$.

$\Rightarrow)$ 
Let $S$ be a commutative $R$-algebra. We have to prove that the morphism 
$\mathcal M^*(S)\to {\tilde N}^*(S)$ is an epimorphism.
The morphism ${\tilde N}\hookrightarrow \mathcal M$ factors through some morphism $\mathcal F_i\to \mathcal M$, because ${\tilde N}$ is an FP-functor ($M$ is a Mittag-Leffler module).
There exists 
$j\geq i$ such that $\Ima[\mathcal M^*(S)\to \mathcal F_i^*(S)]=
\Ima[\mathcal F_j^*(S)\to \mathcal F_i^*(S)]$. 
 $\mathcal M=\ilim{k} {\tilde N_k}$. The morphism $\mathcal F_j\to\mathcal M$ factors through  some morphism ${\tilde N_k}\to \mathcal M$, by Proposition \ref{1.9}.
We have the morphisms
$$\mathcal M^*(S) \to {\tilde N_k}^*(S)\to \mathcal F_j^*(S)\to \mathcal F_i^*(S)\to{\tilde N}^*(S)$$

Then,
$$\aligned 
\Ima[\mathcal M^*(S)\to \mathcal F_i^*(S)]&\subseteq\Ima[{\tilde N_k}^*(S)\to \mathcal F_i^*(S)]\subseteq
\Ima[\mathcal F_j^*(S)\to \mathcal F_i^*(S)]\\ &=\Ima[\mathcal M^*(S)\to \mathcal F_i^*(S)]\endaligned$$
Hence, $\Ima[{\tilde N_k}^*(S)\to \mathcal F_i^*(S)]=
\Ima[\mathcal M^*(S)\to \mathcal F_i^*(S)]$ and
$$\Ima[\mathcal M^*(S)\to {\tilde N}^*(S)]=\Ima[{\tilde N_k}^*(S)\to {\tilde N}^*(S)]\overset{\text{\ref{epi}}}={\tilde N}^*(S).$$

$\Leftarrow)$ Consider a commutative $R$-algebra $S$ and an index $i$. The morphism $\mathcal F_i\to \mathcal M$ factors through  some  morphism ${\tilde N_k}\to\mathcal M$, by Proposition \ref{1.9}. The morphism  ${\tilde N_k}\to\mathcal M$ factors through some morphism  $\mathcal F_j\to \mathcal M$ (because ${\tilde N_k}$ is an FP-functor). We have the morphisms
$$\mathcal M^*(S)\to \mathcal F_j^*(S)\to {\tilde N_k}^*(S)\to \mathcal F_i^*(S)$$
From the commutative diagram
$$\xymatrix{\Ima[\mathcal M^*(S)\to \mathcal F_i^*(S)] \ar@{=}[rr] \ar@{^{(}->}[rd] &  & \Ima[{\tilde N_k}^*(S)\to \mathcal F_i^*(S)] \\ & \Ima[\mathcal F_j^*(S)\to \mathcal F_i^*(S)] \ar@{^{(}->}[ru]&}$$
we have that $\Ima[\mathcal M^*(S)\to \mathcal F_i^*(S)]=
\Ima[\mathcal F_j^*(S)\to \mathcal F_i^*(S)] $.
\end{proof}

\begin{corollary} $M$ is a strict Mittag-Leffler module iff $\mathcal M$ is an $\mathcal R$-submodule 
of some $\mathcal R$-module $\prod_r \mathcal P_r$, where $P_r$ is a finitely presented modules, for every $r$. 

\end{corollary}

\begin{proof} $\Leftarrow)$ Let $P$ be a finitely presented $R$-module. Then,
$$\Hom_{\mathcal R}(\oplus_i \mathcal N_i^*,\mathcal P)\overset{\text{\ref{prop4}}}=\prod_i (N_i\otimes P)=(\prod_i N_i)\otimes P\overset{\text{\ref{1211}}}=
\Hom_{\mathcal R}((\oplus_i\mathcal N_i^*)_{sch},P)$$
Hence, $\Hom_{\mathcal R}(\oplus_i \mathcal N_i^*,\prod_r \mathcal P_r)=\Hom_{\mathcal R}((\oplus_i\mathcal N_i^*)_{sch},\prod_r P_r)$. Consider a monomorphism $\mathcal M\hookrightarrow \prod_r\mathcal P_r$ and the commutative diagram 
$$\xymatrix{ \prod_i (N_i\otimes M)=\Hom_{\mathcal R}(\oplus_i \mathcal N_i^*,\mathcal M)
\ar@{^{(}->}[r] \ar[d] & \Hom_{\mathcal R}(\oplus_i \mathcal N_i^*,\prod_r \mathcal P_r) \ar@{=}[d]
\\ (\prod_i N_i)\otimes M=
\Hom_{\mathcal R}((\oplus_i\mathcal N_i^*)_{sch},\mathcal M) \ar@{^{(}->}[r] & \Hom_{\mathcal R}((\oplus_i\mathcal N_i^*)_{sch},\prod_r\mathcal P_r)     } $$
Then, $ \prod_i (N_i\otimes M)\to (\prod_i N_i)\otimes M$ is inyective. 
By \ref{CML}, $M$ is a strict Mittag-Leffler module.

$\Rightarrow)$ Write $M=\ilim{i} P_i$, where $P_i$ is a finitely presented $R$-module, for any $i$. 
Consider the natural morphism $\mathcal P_i\to \mathcal M$ and let $\tilde P_i$ the image of this morphism. The morphism $\tilde P_i\hookrightarrow \mathcal M$, factors through a morphism
$\tilde P_i\to \mathcal P_{f(i)}$, for some $f(i)>i$, because $\tilde P$ is an FP-functor, by the hypothesis. 
$\mathcal M^*\to \tilde P_i^*$ is an epimorphism, by the hypothesis again. Hence,  $\overline{\mathcal M^*}\to \overline{\tilde P_i^*}$ is an epimorphism. Then, there exists a morphism $s_{f(i)}\colon \mathcal M
\to \mathcal P_{f(i)}$ such that the diagram
$$\xymatrix{\tilde P_i \ar@{^{(}->}[r] \ar@{^{(}->}[d] & \mathcal M \ar[dl]^-{s_{f(i)}}\\ \mathcal P_{f(i)} & }$$
is commutative. Now it is easy to check that the morphism $\mathcal M\to \prod_i \mathcal P_{f(i)}$,
$m\mapsto (s_{f(i)}(m)$ is a monomorphism.

\end{proof}

\begin{theorem} \label{5.8} Let $M$ be an $R$-module. The following statements are equivalent

\begin{enumerate}

\item $M$ is a  Mittag-Leffler module.

\item  The kernel of every morphism
$\mathcal R^n\to\mathcal M$  is isomorphic to a quotient of a module scheme.

\item The kernel of every morphism
$\mathcal N^*\to\mathcal M$ is isomorphic to a quotient of a module scheme, for any  $R$-module $N$.

\end{enumerate}

\end{theorem}

\begin{proof} (1) $\Rightarrow$ (2) Let $g\colon \mathcal R^n\to \mathcal M$ be a morphism of $\mathcal R$-modules.  By Theorem 
\ref{5.6}, $\Ima g$ is an FP-functor.  By Proposition \ref{4.75},  $\Ima g$ is the kernel of an epimorphism $\mathcal P\to\mathcal P'$, where $P$ and $P'$ are finitely presented $R$-modules. By Lemma \ref{4.5}, there exists an exact sequence of morphisms of $\mathcal R$-modules
$$\mathcal V^*\to \mathcal R^n\to \Ima g\to 0$$
Then, $\Ker g=\Ker[\mathcal R^n\to \Ima g]$ is a quotient of $\mathcal V^*$. 

(2) $\Rightarrow$ (3)
Let $L$ be a free module and $g\colon \mathcal L^*\to \mathcal M$ be a morphism of $\mathcal R$-modules. Consider the dual morphism $g^*\colon \mathcal M^*\to \mathcal L$. Observe that 
$g^*$ factors through a finite free direct summand $\mathcal L'$ of $\mathcal L$.
Then, $g$ factors through $\mathcal L'^*$ and $$\Ker g\simeq (\mathcal L/\mathcal L')^*\times \Ker[\mathcal L'^*\to\mathcal M].$$ Hence, $\Ker g$ is a quotient of a module scheme.

Let $N$ be an $R$-module and 
let $f\colon \mathcal N^*\to\mathcal M$  be a morphism of $\mathcal R$-module. Let 
$L$ be a free $R$-module,  
$j\colon L\to N$ be an epimorphism and $N':=\Ker j$. The morphism $\Hom_{\mathcal R}(\mathcal L^*,\mathcal M)=M\otimes L\to M\otimes N=\Hom_{\mathcal R}(\mathcal N^*,\mathcal M)$ is surjective. Hence,
there exists a morphism $g\colon \mathcal L^*\to \mathcal M$
such that the diagram
$$\xymatrix{ \mathcal N^* \ar[r]^-f \ar@{^{(}->}[d]-{j^*} & \mathcal M  \\  \mathcal  L^* \ar[ur]-g &}$$
is commutative. Let $\pi\colon \mathcal V^*\to \Ker g$ be an epimorphism and let $h$ be the composite morphism $\Ker g\hookrightarrow \mathcal L^*\to \mathcal N'^*$. Observe that $\Ker f=\Ker h$. Let
$\mathcal W^*=\Ker (h\circ\pi)$. The morphism $\mathcal W^*=\Ker (h\circ\pi)\overset\pi\to \Ker h=\Ker f$ is an epimorphism, because $\pi$ is an epimorphism.

(3) $\Rightarrow$ (1) Let $f\colon \mathcal R^n\to \mathcal M$ a morphism of $\mathcal R$-module. There exists an epimorphism $\pi\colon \mathcal V^*\to \Ker f$. Then, $\Ima f=\Coker \pi$ is an FP-funtor by \ref{4.8}.
Then, $M$ is a Mittag-Leffler module by \ref{4.6}.

\end{proof}

\begin{proposition} Let $M$ be an $R$-module. The following statements are equivalent

\begin{enumerate}

\item $M$ is a  strict Mittag-Leffler module.

\item  The cokernel of every morphism
$\mathcal M^*\to\mathcal R^n$  is isomorphic  to an $\mathcal R$-submodule of a quasi-coherent module.

\item The cokernel of every morphism
$\mathcal M^*\to\mathcal N$ is isomorphic to an $\mathcal R$-submodule of a quasi-coherent module, for any  $R$-module $N$.

\end{enumerate}

\end{proposition}

\begin{proof} (1) $\Rightarrow$ (2) Let $f\colon \mathcal M^*\to \mathcal R^n$ a morphism. Consider the dual morphism $f^*\colon \mathcal R^n\to \mathcal M$. The morphism $\mathcal M^*\to (\Ima f^*)^*$ is an epimorphism, by \ref{5.6}. There exists an exact sequence of morphisms
$$\mathcal V^*\to \mathcal R^n\to \Ima f^*\to 0$$
by \ref{5.8}. Then, $0\to (\Ima f^*)^*\to \mathcal R^n\to \mathcal V$ is exact and $\Coker f=\Coker[(\Ima f^*)^*\to \mathcal R^n]$ is an $\mathcal R$-submodule of $\mathcal V$.

(2) $\Rightarrow$ (3) Dually, proceed as (2) $\Rightarrow$ (3) in \ref{5.8}.

(3) $\Rightarrow$ (2) It is obvious.

(2) $\Rightarrow$ (1) Let $\phi\colon \mathbb P\to \mathbb P'$ be a morphism of $\mathcal R$-modules. It is easy to check that
$$(\Coker \phi)^*=\Ker\phi^*.$$

 Let $f\colon\mathcal R^n\to\mathcal M$ a morphism of $\mathcal R$-modules. Consider the dual morphism $f^*\colon \mathcal M^*\to \mathcal R^n$. $\Coker f^*$ is an FP-functor, by \ref{4.8}. By the hypothesis, there exists a monomorphism $i\colon \Coker f^*\hookrightarrow \mathcal V$. By \ref{4.75}, there exist finitely presented $R$-modules and an exact sequence of morphisms
$$0\to \Coker f^*\to \mathcal P\to \mathcal P'\to 0$$
By \ref{3.6}, the natural morphism $\mathcal P^*\to (\Coker f^*)^*=\Ker f$ is an epimorphism. Let $h$ be the composite morphism $\mathcal P^*\to \Ker f\to \mathcal R^n$. Then, $\Ima f=\Coker h$, which is an FP-functor by \ref{4.8}, and $\Ker h^*=\Ima f^*$. Observe that 
$$(\Ima f)^*=(\Coker h)^*=\Ker h^*=\Ima f^*$$
Hence, the morphism $\mathcal M^*\to (\Ima f)^*$ is an epimorphism and $M$ is a strict Mittag-Leffler by \ref{5.6}.

\end{proof}


\begin{thebibliography}{99}

\bibitem{Amel} \textsc{\'Alvarez, A., Sancho, C., Sancho, P.,}
\textit{\!Algebra schemes and their representations}, J. Algebra
{\bf 296/1} (2006) 110-144.


\bibitem{Amel2} \textsc{\'Alvarez, A., Sancho, C., Sancho, P.,}
\textit{\!Characterization of Quasi-Coherent Modules  that
are Module Schemes}, Communications in Algebra (2009),37:5,1619  1621.

\bibitem{gabriel} \textsc{Demazure, M.; Gabriel, P.,}
\textit{\! Introduction to Algebraic Geometry and Algebraic
Groups}, Mathematics Studies {\bf 39}, North-Holland, 1980.


\bibitem{Garfinkel}
 \textsc{Garfinkel, G.S.},  \textit{\! Universally torsionless and trace modules}, Trans. Amer. Math. Soc. 215 (1976) 119-144.


\bibitem{Grot} \textsc{Grothendieck, A.} \textit{\! EGA, III.}  Math. Inst. Hautes Etudes Scient. 11 (1961)



\bibitem{RaynaudG} \textsc{Gruson,L., Raynaud, M.,} \textit{\!Crit\`eres de platitude et de projectivit\'e}, Inventiones math. 13, 1-89 (1971).


\bibitem{Pedro} \textsc{Navarro, J., Sancho C., Sancho, P.,} \textit{\! Affine functors and duality} 2012 arXiv:0904.2158v4


\bibitem{Raynaud} \textsc{Raynaud, M.,} \textit{\! Flat modules in algebraic geometry} Compositio Mathematica {\bf 24} n1 1972.

\bibitem{Pedro2} \textsc{Sancho C.,Sancho F.,Sancho, P.,} \textit{\! Geometric characterization of flat modules}  2017 arXiv:1609.08327v4

\bibitem{Pedro3} \textsc{Sancho C.,Sancho F.,Sancho, P.,} \textit{\! Flat SML modules and reflexive functors}
2017 arXiv:1609.08327v4

\bibitem{stacks-project}  \textit{\! Stacks Project}  Version e7f99af, compiled on Oct 25, 2016.

\end{thebibliography}
\end{document}